\documentclass[graybox]{svmult}

\usepackage{color}

\usepackage{mathptmx}

\usepackage{helvet}
\usepackage{makeidx}
\usepackage{multicol}
\usepackage{footmisc}
\usepackage{amssymb}

\def\QED{\hfill\vrule height 1.5ex width 1.4ex depth -.1ex \vskip20pt}

\begin{document}
\title*{On inversions and Doob $h$-transforms of linear diffusions}
\author{L. Alili,   P. Graczyk and  T. \.Zak}
\institute{Name of First Author \at Larbi Alili, Department of
Statistics, The University of Warwick, CV4 7AL, Coventry, UK.
\email{L.alili@Warwick.ac.uk}
\and  Name of Second Author \at Piotr Graczyk, D\'epartement de Math\'ematiques, Universit\'e d'Angers,
UFR Sciences, 2 boulevard Lavoisier, 49045 Angers cedex 01.
\email{graczyk@univ-angers.fr}
\and Name of Third Author
\at Tomasz \.Zak,
Institute of Mathematics and Computer Science, Wroc{\l}aw University of
Technology, Wybrze\.ze Wyspia\'nskiego 27, 50-370 Wroc{\l}aw, Poland.
\email{Tomasz.Zak@pwr.edu.pl}}
%
%

\maketitle

\abstract{ Let $X$ be a regular linear diffusion whose state space is an open
interval $E\subseteq \mathbb{R}$. We consider the  dual diffusion $X^*$
  whose probability law is obtained as a Doob $h$-transform of the law of $X$,  where $h$ is a positive harmonic function for the infinitesimal generator of $X$ on $E$.
    We provide a construction of $X^*$ as a deterministic inversion $I(X)$ of $X$, time
changed with some random clock. Such  inversions  generalize the Euclidean inversions that intervene when $X$ is a Brownian motion. The important case  where  $X^*$ is  $X$ conditioned to stay above some fixed level is included.
The families of deterministic inversions are given explicitly for the Brownian motion with drift, Bessel processes
and the 3-dimensional hyperbolic Bessel process.}





\section{Motivations and main results}
One of the main incentives for carrying out this work was the paper [27] where Marc Yor studied
involutions which are different from the classical inversion with respect to the unit sphere on $R^d$ , $d \geq 1$. We had with him interesting discussions and received from him many encouragements to finalize this work, particularly during the conference {\it Analyse harmonique et Probabilit\'es, Angers 2012}. We would like to thank him for that and for the dear time and invaluable advices he gave to the first and second authors.

Since then he passed away. May his soul rest in peace!

The main objective of this paper is to study  analytical aspects of the stochastic Doob duality.
 We elucidate a striking equivalence between stochastic Doob duality  of one-dimensional diffusions $X$ and $X^*$ and a simple analytical transformation
$I(X)$ of trajectories of a diffusion via a deterministic inversion $I$.

The construction was known for the case
where $X$ is a $3$-dimensional Bessel process started at a positive $X_0$,   the dual process $X^*$ is a Brownian motion killed when it hits $0$ and
  the inversed process  is $1/X$, which  is a Brownian motion conditioned via a Doob $h$-transform to stay positive. The processes  $X^*$
  and $1/X$ coincide up to a time change, see \cite{ry99}.

  It was also known \cite{ry99} that the three-dimensional hyperbolic Bessel process can be realized via a Doob transform as a Brownian motion with negative unit drift conditioned to stay positive.   This work  was inspired by the search and  discovery of  an inversion
  $$I(X_t)=\frac{1}{2} \ln \coth X_t$$
  of the $3$-dimensional hyperbolic Bessel process $(X_t)$.  When $I(X_t)$  is appropriately time changed, we obtain a Brownian motion with negative unit drift,
  see Section \ref{HypBes}.

 The main result of this paper says  that an analytical inversion $I$ can be constructed for any pair of dual linear diffusions  $X$ and $X^*$, see Theorem \ref{equality-in-distribution}. \\
  A direct application of this result is a better understanding of the conditioned diffusions $X^*$: they are obtained, up to a time change, as an analytical transformation $I(X)$ of the original diffusion $X$. Both the families of inversions and the random clocks involved in the construction  have interesting features and deserve their own right of mathematical interest.

Our original motivations for the search of deterministic inversions of stochastic processes come from potential theory, where
 a crucial role is played by the Kelvin transformation, related to the  inversion
 with respect to the unit sphere $I(x)=x/\|x\|^2$, see e.g. \cite{BogdanLN} and  \cite{Bogdan-Tom}.  One of the other reasons why we worked on this topic
 is a strong need of such analytical tools to develop the potential theory of various important processes, e.g.   hyperbolic Brownian motions and hyperbolic Bessel  processes.

Taking into account the results of  \cite{Bogdan-Tom} for stable processes, it is natural to ask whether  such analytical constructions of  conditioned processes should also be available for one-dimensional self-similar processes. In a work in progress, this question and some related topics  are studied in collaboration with L. Chaumont.

 The paper is organized as follows.
  In Section 2 we introduce basic notions and notations on diffusions $X$ with a state space $E$ and  we explain precisely the objectives of the paper.
 We start Section 3  with the construction of a family of inversions associated with a diffusion $X$.  The construction involves a reference scale function $s$ and not the speed measure of $X$. We note that the inversion of the state space $E $ in the direction of $s$ is uniquely characterized  by the fixed point $x_0$. Also, among the set of inversions in the direction of $s$, the $s$-inversion  with fixed point $x_0$ is  uniquely characterized by the associated   positive harmonic function $h$. Thus, the family of inversions we obtain is a one parameter family of involutions indexed by the fixed point $x_0$.   In Section 4, we state  and prove our main result. That gives the path construction of $X^*$ in terms of the inverse of $X$ in the direction of $s$ with respect to a point $x_0\in E$. Section 5 is devoted to
applications. We point out in Corollaries \ref{hypBessel} and \ref{Bessel} some new results that we obtain for Bessel processes and the hyperbolic Bessel process of dimension $3$.

\section{Preliminaries on dual processes and inversions} Let $X:=(X_t, t\leq \zeta)$ be a regular diffusion with life time $\zeta$ and state space  $E=(l,r)\subseteq \mathbb{R}$ which is defined on complete probability space $(\Omega, (\mathcal{F}_t)_{t\geq 0}, \mathbb{P})$. Unless otherwise specified, we assume that $X$ is killed, i.e. sent to a cemetery point $\Delta$, as soon as it hits one of the boundaries; that is $\zeta=\inf\{s,\  X_s=l \ \hbox{or}\ X_s=r \}$
 with the usual convention
$\inf\{\emptyset\}=+\infty$.

Our objectives in this paper are summarized as follows. Given a positive function $h$ which is harmonic for the infinitesimal generator $L$ of $X$, i.e. $Lh=0$, we
give an explicit construction of the dual $X^*$ of $X$ with respect to $h(x)m(dx)$ where  $m(dx)$ is the speed measure of $X$.  The distribution of $X^*$  is obtained by a Doob $h$-transform change of measure of the distribution of $X$. We shall see that $X^*$ is either the process itself, i.e. the process is self-dual, or the original diffusion conditioned to have opposite behaviors at the boundaries when started from a specific point $x_0$ in the state space; this is explained in details in Proposition \ref{Proposition4} below. We refer to the original paper \cite{Doob-1957} by J.L. Doob for $h$-transforms and to \cite{Salm-Boro} where this topic is surveyed. The procedure consists in first constructing the inverse of the diffusion with respect to a point $x_0\in E$ which is a deterministic involution of the original diffusion. Time changing then with an appropriate clock gives a realization of the dual process.
  In order to say more, let us fix the mathematical setting.   Suppose that $X$ satisfies the s.d.e.
 \begin{equation}\label{sde-x}
X_t=X_0+\int_0^t \sigma(X_s)dW_s+\int_0^t b(X_s)ds,\quad t<\zeta ,
\end{equation}
where $X_0\in E$, $(W_t, t\leq \zeta)$ is a standard Brownian motion and  $\sigma,\, b:E\rightarrow \mathbb{R}$ are measurable real valued functions. Assume that $\sigma$ and $b$ satisfy the Engelbert-Schmidt conditions
\begin{equation}\label{condition-1}
\sigma\neq 0 \ \hbox{and}\  1/\sigma^2,\, b/\sigma^2\in L^1_{loc}(E),
 \end{equation}
 where $L^1_{loc}(E)$ is the space of locally integrable functions on $E$. Condition (\ref{condition-1}) implies that (\ref{sde-x}) has a unique solution in law, see Proposition 5.15 in \cite{Karatzas-Shreve}. We write $(\mathcal{F}_t^X, t\geq 0)$ for the natural filtration generated by $X$ and denote by  $\mathcal{DF}(E)$ the set of diffusions  satisfying the aforementioned conditions.  For background on diffusion processes, we refer to
(\cite{Assing-Schmidt}, \cite{Salm-Boro}, \cite{Ito-Mc}, \cite{Karatzas-Shreve}, \cite{Pinsky-95}--\cite{RW87}).

Let $X\in \mathcal{DF}(E)$. For  $y\in E$, let $H_y=\inf\{
t>0;\: X_t=y\}$ be the first hitting time of $y$ by $X$. Recall that the scale function of $X$ is any continuous strictly increasing function on $E$ satisfying
\begin{equation}\label{hitting}
\mathbb{P}_{x}(H_{\alpha}<H_{\beta})=(s(x)-s(\beta))/(s(\alpha)-s(\beta))
\end{equation}
for all $l<\alpha<x<\beta<r$. This is a reference function which is strictly increasing and given
modulo an affine transformation by $s(x)=\int_c^x \exp {(-2\int_c^z b(r)/\sigma^2(r)dr)}dz $ for some $c\in E$.
 For convenience, we distinguish, as in Proposition 5.22 in \cite{Karatzas-Shreve}, the following four different subclasses of diffusions which exhibit different forms of inversions, i.e. mappings $I:E\rightarrow E$ such that $I\circ I(x)= x$, for all  $x\in E$, and $I(E)=E$ (for a more  precise definition of an $s$-inversion, see Definition \ref{DefInv}). We say that $X\in \mathcal{DF}(E)$ is of
 \begin{itemize}
\item Type 1  if $-\infty<s(l)$ and $s(r)<+\infty$;
\item Type 2  if $-\infty<s(l)$ and $s(r)=+\infty$;
\item Type 3  if $s(l)=-\infty$ and $\ s(r)<+\infty$;
\item Type 4  if $s(l)=-\infty$ and $s(r)=+\infty$.
\end{itemize}
Type 4 corresponds to recurrent diffusions while Types 1--3 correspond to transient ones.
 Recall that the infinitesimal generator of $X$ is given by $Lf=(\sigma^2/2)f''+bf'$ where $f$ is in the domain  $\mathcal{D}(X)$ which is appropriately defined for example in \cite{Ito-Mc}.
 For $x_0\in E$ let $h$ be the unique positive harmonic function for $L$ satisfying $h(x_0)=1$ and either
\begin{eqnarray}\nonumber
h(l) =
 \left\{
\begin{array}{lll}
   1/h(r) &\hbox{if}\ X\ \hbox{is of type 1 with}\ 2s(x_0)\neq s(l)+s(r); \\
   1 &\hbox{if}\ X\ \hbox{is of type 1 and}\ 2s(x_0)=s(l)+s(r)\ \hbox{or}\ \hbox{type 4};\\
  0 &\hbox{if}\ X\ \hbox{is of type 2};  \\
\end{array}
\right.
\end{eqnarray}
or
\begin{eqnarray}
h(r) =
\begin{array}{lll}
   0 &\hbox{if}\  X\ \hbox{is of type 3.} \nonumber\\
\end{array}
\end{eqnarray}
If $X$ is of type 4  or of type 1 with $2s(x_0)=s(l)+s(r)$ then  $h\equiv 1$ otherwise $h$ is specifically given by (\ref{harmonic-h-general}) which is displayed in Section 3 below.

Let $X^*$ be the dual of $X$, with respect to $h(x)m(dx)$, in the following classical sense. For all $t>0$ and all Borel functions $f$ and $g$, we have
$$\int_{E}f(x)P_tg(x) h(x)m(dx)=\int_{E}g(x)P_t^*f(x) h(x)m(dx) $$
where $P_t$ and $P^*_t$ are the semigroup operators of $X$ and $X^*$, respectively.
The probability law of $X^*$ is related to that of $X$ by a Doob $h$-transform, see for example \cite{GJ-Y}. To be more precise,  assuming that  $X^*_0= x\in E$  then the distribution $\mathbb{P}_{x}^{*}$ of $X^*$ is obtained  from the distribution $\mathbb{P}_{x}$ of $X$ by the change of measure
\begin{equation}\label{Doob-h-transform}
d\mathbb{P}_{x}^{*}|_{\mathcal{F}_{t}}=\frac{h(X_{t})}{h(x)}
d\mathbb{P}_{x}|_{\mathcal{F}_{t}}, \quad t< \zeta.
\end{equation}
We shall denote by $\mathbb{E}_{x}^{*}$ the expectation under the probability measure $\mathbb{P}_{x}^{*}$; $X^*$ has  the  infinitesimal generator
$L^{*}f=L(hf)/h$ for $f\in \mathcal{D}(X^*)=\{g:E\rightarrow E,\ hg \in \mathcal{D}(X) \}$. The two processes $(h(X_{t}), t\leq \zeta)$ and $(1/h(X^*_t), t\leq  \zeta^*)$ are continuous local martingales. We shall show that the former process can be realized as
 the latter when time changed with an appropriate random clock. Thus, the expression of either  the process $X^*$ or $X$, which are both $E$-valued, in terms of the other, involves the function $I(x)=h^{-1}(1/h(x))$  which is clearly an involution where it is well defined. We will show in Proposition \ref{Prop-Involutions-0}
that $I : E\rightarrow E$.

 For general properties of real valued involutions, we refer to  \cite{Wiener-al-1988} and  \cite{Wiener-al-2002}. Now, indeed, an intuitive formulation of our main result is that, when $h$ is not constant, the processes $(1/h(X^*_t), t\leq \zeta^*)$ and $(h(X_{\tau_t}), t\leq A_{\zeta})$, if both are started at $x_0$,   have the same law,  where, for $t>0$,
 \begin{equation}\label{defA}
 A_t:=\int_0^t I'^2(X_s)\sigma^2(X_s)/ \sigma^2\circ I(X_s)\ ds,
 \end{equation}
 $\tau_t$ is the inverse of $A_t$, and $I'$ stands for the derivative of $I$. Two interesting features of the involved clocks are described as follows. First,  $\tau_t=A^*_t$ where $A^*_t$ is defined as above with $X$ replaced by $X^*$. Second, $\zeta^*$ (resp. $\zeta$) and $A_{\zeta}$ (resp. $A^*_{\zeta^*}$) have the same distribution; these new identities in distribution for killed diffusions resemble the Ciesielski-Taylor  and Biane identities, see \cite{Biane} and \cite{Ciesielski-Taylor-1962}.
 We call the process $(I(X_t), t<\zeta)$ the inverse of $X$ with respect to $x_0$. We know that Doob $h$-transforming $X$ amounts to conditioning it to behave in a particular way at the boundaries. Our construction sheds light on the exact behaviour of the Doob $h$-transformed process at the boundaries.

In the transient case, the general construction of $X^*$ from $X$ discovered by M. Nagasawa,   in \cite{Nagasawa},  applies to linear diffusions; see also \cite{Meyer-68} and  \cite{ry99}. We mention that this powerful method is used by  M.J. Sharpe in \cite{Sharpe}  and by D. Williams in \cite{Williams} for the study of path transformations of some diffusions. While the latter path transform involves time reversal from cooptional times, such as last passage time,  the construction we present here involves only deterministic inversions and time changes with random clocks of the form (\ref{defA}). Although we only consider one dimensional diffusions in this paper, inversions of stable processes and Brownian motion in higher dimensions are studied in \cite{Bogdan-Tom} and  \cite{Yor-85}, respectively.

\section{Conditioned diffusions and inversions}
  Let $X\in \mathcal{DF}(E)$. To start with, assume that $X$ is of type 1 and let $h:E\rightarrow \mathbb{R}_+$ be a positive harmonic function for the infinitesimal generator $L$ of $X$  satisfying $0<h(l)<h(r)<\infty$. Let  $X^*$ be the dual of $X$ with respect to $h(x)m(dx)$.

  We are ready to state the following result which motivates the construction of inversions; to our best knowledge, the described role of
	 the $h$-geometric mean $x_0$ and the $h$-arithmetic mean  $x_1$, which are defined below, for $X$ and $X^*$, has not been known.

\begin{proposition}\label{Proposition1}   Suppose that $X$ is of type 1. The following assertions hold true.

 \begin{itemize}
 \item[1)]  There exists a unique $x_0\in E$ such that $h^2(x_0)=h(l)h(r)$. We call $x_0$ the $h$-geometric mean of $\{ l, r \}$. Furthermore, for $x\in E$, we have
$
\mathbb{P}_{x}(H_{l}<H_{r})= \mathbb{P}_{x}^*(H_{r}<H_{l})$ if and only if $x=x_0$.

\item[2)] There exists a unique $x_1\in E $ such that $2h(x_1)=h(l)+h(r)$. We call $x_1$ the $h$-arithmetic mean of $\{ l, r \}$. Furthermore, for $x\in E$, we have
$
\mathbb{P}_{x}(H_{l}<H_{r})= \mathbb{P}_{x}(H_{r}<H_{l})=1/2$
if and only if $x=x_1$.
\end{itemize}
\end{proposition}
\begin{proof} 1) Since $h$ is continuous and monotone, because it is an affine function of $s$,  if $h$ is increasing (resp. decreasing) then the inequality $h(l)<\sqrt{h(l)h(r)}<h(r)$ (resp. $h(r)<\sqrt{h(l)h(r)}<h(l)$) implies the existence and uniqueness of $x_0$. Because $-1/h$ is a scale function for $X^*$, see for example \cite{Salm-Boro}, applying (\ref{hitting})  yields
\begin{eqnarray}
\mathbb{P}_{x}(H_{l}<H_{r})=\left(h(x)-h(r) \right)/\left(h(l)-h(r) \right) \nonumber
\end{eqnarray}
and
\begin{eqnarray}
\mathbb{P}_{x}^*(H_{l}>H_{r})= \left(1/h(x)-1/h(l) \right)/\left(1/h(r)-1/h(l)\right). \nonumber
\end{eqnarray}
These are equal if and only if $x=x_0$. \\

2) The proof is omitted since it is very similar.
\QED
\end{proof}

As $h$ is monotone, we find $r=h^{-1}\left(h^2(x_0)/h(l) \right)$. This expression of $r$ in terms of $x_0$ and $l$ allows us
to introduce the mappings we are interested in. The last formula exhibits the function $I:x\rightarrow h^{-1}\left(h^2(x_0)/h(x) \right)$ which is well defined on $E$ by monotonicity of $h$. Clearly, $I$ is a decreasing involution of $E$ with fixed point $x_0$.  Next, observe that $h\circ I \circ h^{-1}:x\rightarrow h^2(x_0)/x$ is the Euclidian inversion with fixed point $h(x_0)$.\\

  We return now to the general case and assume that $X$ is of one of the types 1--4. Our aim is to determine the set of all involutions associated to $X$ which lead to the set of M\"obius real involutions
$$\mathcal{MI}:=\left\{\omega:\mathbb{R}\backslash \{a/c\}\rightarrow \mathbb{R}\backslash \{a/c\}; \omega(x)=\frac{ax+b}{cx-a}, a^2+bc> 0, \ a, b, c\in \mathbb{R} \right\}. $$
Note that the condition $a^2+bc>0$ for $\omega\in \mathcal{MI}$ ensures that $\omega$, when restricted to either of the intervals $(-\infty, a/c)$ and $(a/c, +\infty)$, is a decreasing involution. Let us settle the following definition.
\begin{definition}\label{DefInv}
 Let $s$ and $s^{-1}$ be, respectively, a reference scale function for $X$ and its inverse function, and $x_0\in E$. A mapping $I:E\rightarrow E$ is called the inversion in the direction of $s$, or $s$-inversion, with fixed point $x_0$ if the following hold:
\begin{itemize}
\item[1)] $I\circ I(x)=x$ for $x\in E$;
\item[2)] $s\circ I \circ s^{-1}\in \mathcal{MI}$;
\item[3)] $I(E)=E$;
\item[4)] $I(x_0)=x_0$.
\end{itemize}
If $s\circ I \circ s^{-1}$ is the Euclidian reflection in $x_0$ then $I$ is called the $s$-reflection in $x_0$.
\end{definition}
Since $I$ is defined on the whole of the open interval $E$, it is necessarily continuous. This, in turn, implies that it is  a decreasing involution such that $I(l)=r$. The objective of our next result is to show the existence of the inversion of $E$ in the direction of $s$ in case when $s$ is bounded on $E$ i.e. for diffusions of type 1.

 \begin{proposition}\label{Prop-Involutions-0}  Let $x_0\in E$ and assume that $s$ is bounded on $E$. Then, the following assertions hold.
\begin{itemize}
\item[ 1)] The inversion of $E$ in the direction of $s$ with fixed point $x_0$ is given by
\begin{eqnarray}\label{involution-expression-2}
 I(x)=
 \left\{
\begin{array}{lll}
  s^{-1}\left( s^2(x_0)/s(x)\right) &\hbox{if}\ s^2(x_0)= s(l)s(r);& \\
  s^{-1}\left(2s(x_0)-s(x) \right) &\hbox{if}\ 2s(x_0)=s(l)+s(r);\\
  s^{-1}\left(  (s(x)+a)/(b s(x)-1)\right) &\hbox{otherwise,}  & \\
\end{array}
\right. \nonumber
\end{eqnarray}
where
$$
a=\left(2s(l)s(r)-s(x_0)\left(s(l)+s(r) \right)\right)\left(s^2(x_0)-s(l)s(r)\right)^{-1}s(x_0)\nonumber
$$
and
$$
b=\left(2s(x_0)-\left(s(l)+s(r) \right)\right)\left(s^2(x_0)-s(l)s(r)\right)^{-1}.
$$
\item[ 2)] If $2s(x_0)\neq s(l)+s(r)$ then  $I=h^{-1}(1/h)$ where
\begin{eqnarray}\label{harmonic-h}
 h(x)=
 \left\{
\begin{array}{lll}
(bs(x)-1)/(bs(x_0)-1) &\hbox{if}\ s^2(x_0)\neq s(l)s(r); \nonumber \\
  s(x)/s(x_0) &\hbox{otherwise}.& \nonumber\\
\end{array}
\right.
\end{eqnarray}
Furthermore, $h$ is continuous, strictly monotonic and satisfies $h\neq 0$ on $E$.
\end{itemize}
 \end{proposition}
 \begin{proof} 1) We shall first assume that $s^2(x_0)\neq s(l)s(r)$. We look for $I$ such that $s\circ I \circ s^{-1}(x)=(x+a)/(bx-1)$ where $a$ and $b$ are reals satisfying $ab+1\neq 0$. Since the images of $l$ and $x_0$  by $I$ are respectively $r$ and $x_0$, we get the following linear system of equations
\begin{eqnarray}\label{system}
 \left\{
\begin{array}{lll}
   bs^2(x_0)-a &=& 2s(x_0);\\
 b s(l)s(r)-a &=& s(l)+s(r).  \\
\end{array}
\right.
\end{eqnarray}
Solving it yields $a$ and $b$. We need to show that $ab+1\neq 0$. A manipulation of the first equation of (\ref{system}) shows that $1+ab=(bs(x_0)-1)^2$. In fact, we even have the stronger fact that $s(x)\neq 1/b$ on $E$ which is seen from $1/b>s(r)$ if $2s(x_0)>s(l)+s(r)$ and $1/b<s(l)$ if $2s(x_0)<s(l)+s(r)$. Finally, if $s^2(x_0)=s(l)s(r)$ then clearly $I(x)=s^{-1}(s^2(x_0)/s(x))$.

 2) Assume that $2s(x_0)\neq s(l)+s(r)$. Let us first consider the case $s^2(x_0)\neq s(l)s(r)$. Setting  $h(x)=(s(x)-1/b)/\delta$  we then obtain  $$h^{-1}(1/h(x))=s^{-1}\left((s(x)+(b\delta^2-1/b))/(bs(x)-1) \right).$$ Thus, the equality $I(x)=h^{-1}(1/h)$ holds if and only if $\delta=\pm \sqrt{1+ab}/b$ which, in turn, implies that $h(x)=\pm (bs(x)-1)/(bs(x_0)-1)$. Since $h$ is positive, we take the solution with plus sign.  Since $h$ is an affine transformation of $s$, it is strictly monotone and continuous on $E$. Finally, because $s\neq 1/b$, as seen in the proof of 1), we conclude that $h$ does not vanish on $E$. The case $s^2(x_0)= s(l)s(r)$ is completed by observing that this corresponds to letting $b\rightarrow\infty$ and $\delta=s(x_0)$ above which gives the desired expression for $h$. \QED
 \end{proof}
 Now, we are ready to fully describe the set of inversions associated to the four types of diffusions described in the introduction. The proof of the following result is omitted, keeping in mind that
 when $s$ is unbounded on $E$, by approximating $E$ by a family of intervals  $(\alpha, \beta)\subset E$ where $s$ is bounded,   using continuity and letting $\alpha\rightarrow l$ and  $\beta\rightarrow r$ we obtain an expression for $I$.

 \begin{proposition}\label{kinds-involutions} All kinds of inversions of $E$ in the direction of $s$ with fixed point $x_0\in E$ are described as follows.
\begin{itemize}
 \item[1)] $X$ is of type 1 with $2s(x_0)\neq s(l)+s(r)$ then the inversion is given in Proposition~\ref{Prop-Involutions-0}.
 \item[2)] $X$ is of type 2 then we have
\begin{equation}\label{involution-expression-0}\nonumber
I(x)=s^{-1} \left(  s(l)+(s(x_0)-s(l))^2/(s(x)-s(l)) \right).
\end{equation}
\item[3)] $X$ is of type 3 then we have
\begin{equation}\label{involution-expression-1}\nonumber
I(x)=s^{-1} \left(  s(r)-(s(r)-s(x_0))^2/(s(r)-s(x)) \right).
\end{equation}
\item[4)] $X$ is of type 4  or type 1 with $2s(x_0)=s(l)+s(r)$ then  $I$ is the $s$-reflection in $x_0$.
\end{itemize}
\end{proposition}

\begin{remark}
Observe  that the  $s$-inversions  described in Proposition \ref{kinds-involutions} solve  $G(x,y)=0$ in $y$, where $G$ is the symmetric function  $G(x,y)=As(x)s(y)-B(s(x)+s(y))-C$ for some reals $A$, $B$ and $C$. This is in agreement with the fact that $I$ is an involution, see \cite{Wiener-al-2002}.
\end{remark}
\begin{remark}
The inversion in the direction of $s$ with fixed point $x_0$ does not depend on the particular choice we make of $s$. Tedious calculations show that the inversion of $E$ in the direction of $s$ is invariant under a M\"obius transformation of $s$.
\end{remark}

Going back to Proposition \ref{Proposition1} we can express the inversions of Proposition \ref{kinds-involutions} in terms of the harmonic function $h$ instead of the reference scale function $s$. For that we need to compute the positive harmonic function $h$ described in the introduction for each of the types 1--4 of diffusions. We easily get
   \begin{eqnarray}\label{harmonic-h-general}
 \rule{0,7cm}{0cm}h(x)=
 \left\{
\begin{array}{lll}
\frac{bs(x)-1}{bs(x_0)-1} &\hbox{if}\  X\ \hbox{is of type 1 and}\  2s(x_0)\neq s(l)+s(r); & \\
  1 &\hbox{if}\ X\ \hbox{is of type 1 and}\ 2s(x_0)=s(l)+s(r)\ \hbox{or}\ \hbox{type 4};&\\
  \frac{s(x)-s(l)}{s(x_0)-s(l)} &\hbox{if}\  X\ \hbox{is of type 2};&\\
  \frac{s(r)-s(x)}{s(r)-s(x_0)} &\hbox{if}\ X\ \hbox{is of type 3}.&
\end{array}
\right.
\end{eqnarray}
 Note that the case where $X$ is of type 1 and $x_0$ is the $s$-geometric mean is covered in the first case by letting $b\rightarrow \infty$ to obtain $h(x)=s(x)/s(x_0)$.  In the following result, which generalizes Proposition \ref{Proposition1}, we note that
the first assertion could serve as the probabilistic definition for $s$-inversions.

\begin{proposition}\label{Proposition4} The following assertions hold true.
\begin{itemize}
\item[1)] A function $I:E\to E$ is the $s$-inversion  with fixed point $x_0\in E$ if and only if $I(E)=E$ and for all $x\in E$
\begin{equation}\label{Condition-involution}
\mathbb{P}_{x_0}(H_{x}<H_{I(x)})= \mathbb{P}_{x_0}^*(H_{I(x)}<H_x)
\end{equation}
where $\mathbb{P}^*$ is the distribution of the Doob transform of $X$ by some positive harmonic function $k$.
 Furthermore, the $s$-inversion and the $k$-inversion of $E$ with fixed point $x_0$ are equal and $k=h$.
\item[2)] Let $\mathbb{Q}_{x_0}$ be the probability law of $(I(X_t), t\leq \zeta)$ when $X_0=x_0$. Then formula (\ref{Condition-involution}) holds true when $\mathbb{P}_{x_0}$ is replaced by $\mathbb{Q}_{x_0}$. We call the process $(I(X_t), t\leq \zeta)$ the inverse with respect to $x_0$ of $(X_t, t\leq \zeta)$. The fixed point $x_0$ of the involution $I$ is seen to be the unique level at which the paths of the latter processes intersect.
\end{itemize}
\end{proposition}
\begin{proof} 1) If $x_0$ is the $s$-arithmetic mean of $\{l, r \}$ or $X$ is of type 4 then we are looking for $I:E\rightarrow E$ such that $\mathbb{P}_{x_0}(H_{x}<H_{I(x)})=1/2$. Using (\ref{hitting}) we get $(s(x_0)-s\circ I(x))/(s(x)-s\circ I(x))=1/2$ which gives that $I$ is the $s$-reflection. For the other cases, using (\ref{Condition-involution}) and the fact that $-1/h$ is a scale function for $X^*$,
 we find  that $I(x)=h^{-1}(h(x_0)^2/h(x))$ so that $I$ is an $h$-inversion with fixed point $x_0$. The ``only if'' part is straightforward following a similar reasoning to that of the proof of Proposition \ref{Proposition1} giving  $I$ to be either the $h$-reflection or the $h$-inversion with fixed point $x_0$. \\
 2) The first part is easily seen by using the first assertion. The interpretation for the fixed point $x_0$ follows from the fact that $x_0$ is the unique fixed point of $I$. \QED
\end{proof}
For completeness, we explain now how to define rigorously a diffusion $Y_t$ obtained by conditioning a transient diffusion $X_t$ to hit one boundary of an interval
before another with a prescribed probability $a$. By a  natural definition, it holds if
 for any bounded $\mathcal{F}_t^X$-measurable functional $G$ and $t>0$, we have
\begin{eqnarray}
\mathbb{E}_{x_0}[G(Y_s, s\leq t), t< \zeta]
&=&a \mathbb{E}_{x_0}[G(X_s, s\leq t), t<\zeta|H_l<H_r]+\nonumber\\&&(1-a)\mathbb{E}_{x_0}[G(X_s, s\leq t), t<\zeta|H_r<H_l]\nonumber
\end{eqnarray}
 For  conditioning a transient diffusion to avoid one of the boundaries we refer, for example,  to
(\cite{IIzuka-Maeno-Tomizaki-2006}, \cite{Maeno-2003},  \cite{Salminen-Vallois-Yor-2007}).

 We show in  the following  Proposition  that the dual process $X^*$ can be realized as  $X$ conditioned in the sense of Doob to exits the segment
$[l,r]$ at the endpoints $l$ and $r$ with some specified probabilities.
\begin{proposition}\label{conditioning-avoiding}
 Assume that $X$ is transient  and let $h$  be  given by (\ref{harmonic-h-general}).  Let $p$ be the probability that $X$, when started at $x_0$, exits $[l,r]$ at $l$.  $X$ conditioned to exit $[l,r]$  at $l$ with probability $q=1-p$ is a realization of the dual $X^*$ of $X$ with respect to $h(x)m(dx)$. \end{proposition}
\begin{proof}
By construction, we have $h(x_0)=1$. Assume at first that $X$ is of type 1. Let us decompose $h$, in terms of $h_l$ and $h_r$ which are defined below, as follows
$$h(x)=q^*\frac{h(x)-h(l)}{h(x_0)-h(l)}+p^*\frac{h(r)-h(x)}{h(r)-h(x_0)}:=q^*h_r(x)+p^*h_l(x)$$
where
\begin{eqnarray}\nonumber
 q^*=\frac{h(x_0)-h(l)}{h(r)-h(l)}h(r)=\frac{h^*(x_0)-h^*(l)}{h^*(r)-h^*(l)}=\mathbb{P}_{x_0}^*(H_r<H_l)
\end{eqnarray}
and
\begin{eqnarray}\nonumber
p^*=\frac{h(r)-h(x_0)}{h(r)-h(l)}h(l)=\frac{h^*(r)-h^*(x_0)}{h^*(r)-h^*(l)}=\mathbb{P}_{x_0}^*(H_l<H_r).
\end{eqnarray}
But, we have that $q^*=p$ and $p^*=q$ when $X$ and $X^*$ are started at $x_0$. Thus,  for any bounded $\mathcal{F}_t^X$-measurable functional $G$ and $t>0$, we can write
\begin{eqnarray}
\mathbb{E}_{x_0}^*[G(X_s, s\leq t), t< \zeta]&=&\mathbb{E}_{x_0}[h(X_t)G(X_s, s\leq t), t< \zeta] \nonumber\\
&=&q\mathbb{E}_{x_0}[h_l(X_t)G(X_s, s\leq t), t< \zeta] \nonumber\\
&+&p\mathbb{E}_{x_0}[h_r(X_t)G(X_s, s\leq t), t< \zeta].\nonumber
\end{eqnarray}
Next, since our assumptions imply that $p= \mathbb{P}_{x_0}(H_l<H_r)\in (0,\, 1)$, we have
\begin{eqnarray}
\mathbb{E}_{x_0}[h_l(X_t)G(X_s, s< t), t< \zeta]&=&\mathbb{E}_{x_0}[G(X_s, s\leq t) \frac{\mathbb{P}_{X_t}[H_l<H_r]}{\mathbb{P}_{x_0}(H_l<H_r)}, t< \zeta] \nonumber\\
&=&\mathbb{E}_{x_0}[G(X_s, s\leq t), t<\zeta|H_l<H_r] \nonumber
\end{eqnarray}
where we used the strong Markov property for the last equality. Similarly, for the other term, since $q\in (0,\, 1)$  we get
\begin{eqnarray}
\mathbb{E}_{x_0}[(h_r(X_t)G(X_s, s\leq t), t< \zeta]
=\mathbb{E}_{x_0}[G(X_s, s\leq t), t<\zeta|H_r<H_l].\nonumber
\end{eqnarray}
The last two equations imply our assertion.
Assume now that $h(r)=\infty$. Then $h(l)=0$ and  $\mathbb{P}_{x_0}$-a.s. all trajectories of the process $X$ tend to $l$  and  $p= \mathbb{P}_{x_0}(H_r<H_l)=0$.  We follow \cite{Salminen-Vallois-Yor-2007}
to define $X$ conditioned to avoid $l$ (i.e. never to hit $l$ in a positive time) as follows. For any bounded $\mathcal{F}_t^X$-measurable functional $G$ and $t>0$, we set
\begin{eqnarray}
\mathbb{E}_{x_0}^*[G(X_s, s\leq t), t< \zeta]
&=&\lim_{a\rightarrow r}\mathbb{E}_{x_0}[G(X_s, s\leq t), t<\zeta|H_a<H_l] \nonumber\\
&=&\lim_{a\rightarrow r}\mathbb{E}_{x_0}[G(X_s, s\leq t), t<H_a<H_l]/\mathbb{P}_{x_0}(H_a<H_l) \nonumber\\
&=&\lim_{a\rightarrow r}\mathbb{E}_{x_0}[h_r(X_t)G(X_s, s< t), t<H_a\wedge H_l] \nonumber\\
&=&\mathbb{E}_{x_0}[h(X_t)G(X_s, s\leq t), t<\zeta] \nonumber
\end{eqnarray}
where we used the strong Markov property for the third equality and the monotone convergence theorem for the last one. In this case $\mathbb{P}^*_{x_0}$-a.s  all trajectories of the process $X^*$ tend to $r$ and $p^*= \mathbb{P}_{x_0}^*(H_l<H_r)=0$ which completes the proof of the statement.  The case $h(l)=-\infty$ and $h(r)=0$ can be treated similarly.  \QED
\end{proof}

\begin{remark} From the point of view of Martin boundaries, the functions $h_l$ and $h_r$ which appear in the proof of Proposition \ref{conditioning-avoiding} are the minimal excessive functions attached to the boundary points $l$ and $r$, see (\cite{Salm-Boro}, \cite{Salminen-82}, \cite{Salminen-85}). That is the Doob $h$-transformed processes obtained by  using $h_l$ and $h_r$ tend a.s., respectively, to $l$ and $r$. Harmonic functions having a representing measure with support not included in  the boundary set of $E$ are not considered in this paper since we do not allow killings inside $E$.
\end{remark}

\section{Inversion of diffusions}

  Let $X\in \mathcal{DF}(E) $ and $s$ be a scale function for $X$. For $x_0\in E$, let $I:E\rightarrow E$ be the inversion of $E$ in the direction of $s$ with fixed point $x_0$. Let $h$ be the positive harmonic function specified by (\ref{harmonic-h-general}).
Let $X^*$ be the dual of $X$ with respect to $h(x)m(dx)$. As aforementioned, the distribution of $X^*$ is obtained as a
Doob $h$-transform of the distribution of $X$ by using the harmonic function $h$, as given in (\ref{Doob-h-transform}).  Clearly, if $X$ is of type 1 (resp. of type 2 and drifts thus to $l$, of type 3 and drifts thus to $r$ or of type 4) then $X^*$ is of type 1 (resp. of type 3 and drifts thus  to $r$, of type 2 and  drifts thus to $l$ or of type 4). It is easy to see that the inversions of $E$ in the direction of $s$ given in Proposition \ref{kinds-involutions} are differentiable on $E$. Recall that for a fixed $t<\zeta$, $\tau_t$ is the inverse of the strictly increasing and continuous additive functional $A_t:=\int_0^t I'^2(X_s)\sigma^2(X_s)/ \sigma^2\circ I(X_s)\ ds$; $\tau_t^{*}$ and $A_t^{*}$ are the analogue objects associated to the dual $X^{*}$. Recall that the speed measure $m(dy)=2dy/(\sigma^2 s')$ of $X$ is uniquely determined by
$$
\mathbb{E}_x[H_\alpha\wedge H_\beta]= \int_J G_{J}(x,y)\, m(dy)
$$
where
\begin{eqnarray}\label{killed}\nonumber
G_J(x,y) =c
   (s(x\wedge y)-s(\alpha))(s(\beta)-s(x\vee y))/(s(\beta)-s(\alpha))
\end{eqnarray}
for any $J=(\alpha, \beta)\subsetneq E$ and all $ x,y\in J$, where $c$ is a normalization constant and  $G_J(.,.)$ is the potential kernel density relative to $m(dy)$ of $X$ killed when it exits $J$; see for instance \cite{ry99} and  \cite{RW87}.
Recall that $-1/h$ is a scale function and $m^{*}(dx):=h^2(x) m(dx)$ is the speed measure of $X^*$.  We are ready to state the main result in this paper.
 The proof we give is based on
the resolvent method for the identification of the speed measure, see (\cite{ry99}, \cite{RW87}, \cite{Sharpe}). Other possible methods of proof are commented in Remarks \ref{proof2}
and \ref{proof3}.

\begin{theorem}\label{equality-in-distribution} With the previous setting, let $I$ be the $s$-inversion of $E$ with fixed point $x_0\in E$. Assuming that $X_0,\, X^*_0\in E$ are such that $I(X_0)=X^*_0$ then the following assertions hold true.
\begin{itemize}
\item[1) ]  For all  $t<\zeta$, $\tau_t$ and $A_t^{*}$ have the same distribution.
\item[2) ]  The processes $(X^*_t, t\leq \zeta^*)$ and $(I(X_{\tau_t}), t\leq A_{\zeta})$
have the same law.
\item[3) ]  The processes $(X_t, t\leq \zeta)$ and  $(I(X^*_{\tau_t^{*}}), t\leq A^{*}_{\zeta^*} )$  have the same law.
\end{itemize}
\end{theorem}

\begin{remark}
  Note that if the starting point  is the fixed point of $I$, i.e. $X_0=x_0$, then both processes $X$ and $X^*$ start from $x_0$.  However, our result holds true and is proven in the general case $X_0\in E$ provided that $X^*_0=I(X_0)$.
\end{remark}

\begin{proof}[of Theorem \ref{equality-in-distribution}]
1)  Let $t>0$ be fixed and  set $\eta_t=I(X_{\tau_t})$. Because $\tau_t$ is the inverse of $A_t$, we can write $A_{\tau_t} = t$. Differentiating and extracting the derivative of $\tau_t$ yields
\begin{eqnarray}\label{deriv-tau}\nonumber
\frac{d}{dt} \tau_t = \frac{\sigma^2\circ I (X_{\tau_t})} {I'\,^2 (X_{\tau_t}) \sigma^2  (X_{\tau_t})}.
\end{eqnarray}
Integrating yields
\begin{eqnarray} \nonumber
\tau_t =\int_0^t \left({I'}\sigma/(\sigma\circ I) \right)^{-2}(X_{\tau_s}) ds
=A_t^{\eta}.
\end{eqnarray}
The proof of the first assertion is complete once we have shown that $\eta$ and $X^*$ have the same distribution which will be done in the next assertion.

2) First, assume that $h$ is not constant. In this case, $x\rightarrow -1/h(x)$ is a scale function for $\eta$ since $-1/h\circ I(X_{\tau_t})=-h(X_{\tau_t})$ is a continuous local martingale. Next, let $J=(\alpha, \beta)$ be an arbitrary subinterval of $E$. We proceed by identifying the speed measure of $\eta$ on $J$.  By using the fact that the hitting time $H^\eta_y$ of $y$ by $\eta$ equals $A_{H_{I(y)}}$ for $y\in E$,  we can write
\begin{eqnarray}
\mathbb{E}_{I(x)}\left[ {H^\eta_{\alpha}\wedge H^\eta_{\beta}}\right]
&=&\mathbb{E}_{I(x)}\left[\int_0^{H_{I(\alpha)}\wedge H_{I(\beta)}}dA_r\right] \nonumber\\
&=&\int_{I(\beta)}^{I(\alpha)} G_{I(J)}(I(x), y)I'^2(y)\sigma^2(y)\{ \sigma^2\circ I(y)\}^{-1}\, m(dy) \nonumber\\
&=&2\int_{I(\beta)}^{I(\alpha)}  G_{I(J)}(I(x), y)I'^2(y) \{\sigma^2\circ I(y)s'(y)\}^{-1}\, dy \nonumber\\
&=&2\int_{\alpha}^{\beta} G_{I(J)}(I(x), I(y)) \{ \sigma^2(y)(-h\circ I)'(y)\}^{-1}\, dy.\nonumber
\end{eqnarray}
On the one hand, we readily check that $\{ \sigma^2(y)(s\circ I)'(y)\}^{-1} dy=h^2(y)m(dy)=m^*(dy)$ for $y\in J$. On the other hand, we have
\begin{eqnarray}
G_{I(J)}(I(x), I(y))&=&
   \left(h(I(x)\wedge I(y))-h(I(\beta))\right)\frac{h(I(\alpha))-h(I(x)\vee I(y))}{h(I(\alpha))-h(I(\beta))} \nonumber\\
   &=&\left(-h^*(x\vee y)+h^*(\beta)\right)\frac{-h^*(\alpha)+h^*(x\wedge y)}{-h^*(\alpha)+h^*(\beta)} \nonumber \\
   &=&G_J^{*}(x,y), \nonumber
\end{eqnarray}
where $h^*=-1/h$ and $G^*_J$ is the potential kernel density of $X^*$ relative to $m^*(dy)$. The case when $h$ is constant can be dealt with similarly but by working with $s$ instead of $h$.  This shows that the speed measure of $\eta$ is the same as that of $X^{*}$ in all cases. Now,  since $\tau^X_t=A_t^{\eta}$ we get $\tau^X_t=A_t^{*}$ which, in turn, implies that $A_t^X=\tau_t^{*}$.   Finally, using the fact that $I$ is an involution gives
$$
H^{*}_y
=\inf\{s>0, I(X_{A^{*}_s})=y\}
=\tau_{H_{I(y)}}^{*}=A_{H_{I(y)}}
$$
for  $y\in E$.  The assertion is completed by letting $y$ tend to either of the boundaries to find $\zeta^*=A_{\zeta}$ and $\zeta=A^*_{\zeta^*}$ as desired.

3) The proof is easy using 1) and 2), the fact that $I$ is an involution and time changes.  \QED
 \end{proof}

\begin{remark}
The resolvent method used in the proof of Theorem \ref{equality-in-distribution} suggests that it could be not necessary to suppose that $X$ solves a diffusion s.d.e. We conjecture that  variants of Theorem \ref{equality-in-distribution} are true for "nice" Markov processes. For example, analogue inversions are known for symmetric stable processes, see \cite{Bogdan-Tom}. As mentioned in the introduction, inversions of a more general class of self-similar Markov processes are studied in a joint work  with L. Chaumont.
\end{remark}

\begin{remark}\label{proof2}
 Since $X$ satisfies the s.d.e. (\ref{sde-x}), by   Girsanov's theorem, we see that $X^*$ satisfies $Y_t=X_0^* +\int_0^t \sigma(Y_s)dB_s+ \int_0^t (b+\sigma^2 h'/h) (Y_s)ds$ for $t<\zeta^*$
where $B$ is a Brownian motion which is measurable with respect to the filtration generated by $X^*$ and $\zeta^*=\inf\{s,\  Y_s=l \ \hbox{or}\ Y_s=r \}$, see for example \cite{GJ-Y}. Long calculations show that $\eta$ also satisfies the above s.d.e. which, by Engelbert-Schmidt condition (\ref{condition-1}), has a unique solution in law. This gives a second proof of Theorem \ref{equality-in-distribution}. Note that the use of It\^o's formula for $\eta$ is licit since $I\in\mathcal{C}^2(E)$.
\end{remark}
\begin{remark}\label{proof3}
Another way to view the main statements of Theorem \ref{equality-in-distribution} is the equality of generators
$$
\frac{1}{I'^2(x)}\frac{\sigma^2(x)}{\sigma^2(I(x))}L (f\circ I) (I(x))=\frac{1}{h(x)}L(hf) (x)=L^* f(x)
$$
and
$$
\frac{1}{I'^2(x)}\frac{\sigma^2(x)}{\sigma^2(I(x))}L^*(g\circ I)\: (I(x))=L g(x)
$$
for all $x\in E$,   $g\in \mathcal{D}(X)$ and $f\in \mathcal{D}(X^*)$. However, the main difficulty of this method of proof is the precise description of domains of generators.

\end{remark}
The focus now is on the cases where $X$ is transient and drifts a.s. either to $l$ or to $r$ as $t\rightarrow \zeta$, i.e.  $X$ is  of type 2 or 3. Clearly,  if $X$ is of type 2 (resp. type 3) then $X^*$ is of type 3 (resp. 2). Hence, for our purpose, it is enough to consider the case where $X$ is of type 2, in which case
  formula (\ref{hitting}) gives that $l$ is hit a.s. before $r$.
\begin{corollary}\label{equality-in-distribution-2} If $X\in \mathcal{DF}(E)$ is of  type 2  then  $\zeta =H_l$  and $A_{\zeta^*}^*=A_{H^*_r}^*$ have the same distribution. Furthermore, $A_{H^*_r}^*<\infty$ with probability 1 if and only if
\begin{equation}\label{entrance-condition}
\int_l^c (s(x)-s(l)) m(dx)<\infty, \quad  l<c<r.
\end{equation}
\end{corollary}
\begin{proof} The first claim is a straightforward consequence of Theorem  \ref{equality-in-distribution}. Next,  $\zeta^*=H^*_r$  because $X^*_t \rightarrow r$ a.s. as $t\rightarrow \zeta$. Now by Feller's classification of boundaries,   $\zeta<\infty$ a.s., and hence  $A^*_{H^*_r}<\infty$ a.s., if and only if (\ref{entrance-condition})  holds, see e.g.  p.~745 in \cite{Williams}.

\QED
\end{proof}
So far we made the assumption that attainable boundaries are killing and the process cannot be started from such points. We stress out however that, in the following result, we assume that $l$ is an entrance not an exit (and not an  absorbing or killing) point for the dual diffusion $X^*$.
The following result gives a path construction of the diffusion which we obtain when we apply the time reversal property to $X^*$ (or $X$).
\begin{corollary} \label{time-reversal}  Assume that $X\in \mathcal{DF}(E)$ is of type 2 and satisfies (\ref{entrance-condition}). Introduce the last passage time of $X$ at the fixed point  $x_0\in E$ of the $s$-inversion $I$, i.e.
$$
L_{x_0}=\sup\{t: X_t=x_0 \}.
$$
 Then the time inverted process $(X_{L^*_{x_0}-t}^*, t\leq L_{x_{0}}^*|X_0^*=l)$ and $(I(X^*_{\tau_t^*}), t<A_{H^*_r}^*|X_0=x_0)$ are identical in law.
\end{corollary}
\begin{proof} Let $h$ be given by (\ref{harmonic-h-general}). Then $s^*(x)=-1/h(x)$ is a scale function for $X^*$ and $s^*(l)=-\infty$. Thus, condition (\ref{entrance-condition}) implies that $l$ is an entrance not exit point for $X^*$, see \cite{Jacobsen}.  Furthermore, the process $X^*$, when started at $l$, is  $X$ conditioned never to return to $l$ in a positive time.  Now, on the one hand, we have that  $(X_{t}, t\leq H_l|X_0=x_0)$ and $(X_{L^*_{x_0}-t}^*, t\leq L_{x_{0}}^*|X_0^*=l)$ have the same distribution, see for example Theorem 2.5 in \cite{Williams}.  On the other hand, we know by Theorem  \ref{equality-in-distribution} that $(X_{t}, t< \zeta|X_0=x_0)$ and $(I(X^*_{\tau_t^*}), t<A_{\zeta^*}^*|X_0=x_0)$ have the same law.  \QED
\end{proof}
\begin{remark}  Theorem 2.11 of \cite{MU-2012}  states that if $X\in \mathcal{DF}(E)$ with $s(r)<\infty$ and  $f:E\rightarrow \mathbb{R}$ is a
non-negative Borel function, then it holds that $\int_0^{\zeta}f(X_s)ds<\infty$ a.s., on the event   $\{\lim_{t\rightarrow \zeta} X_t=r\}$, if and only if $$\int_c^r (s(r)-s(x)) f(x)m(dx)<\infty, \quad l<c<r.$$
Keeping the setting of Corollary  \ref{equality-in-distribution-2} and applying the above to  $X^*$,
with
$$ m^*(x)=(s(x)-s(l))^2m(dx)\quad \hbox{and} \quad f(x)=I'^2(x)\frac{\sigma^2(x)}{\sigma^2\circ I(x)}, $$
we obtain the necessary and sufficient condition
$$\int_c^r (s^*(r)-s^*(x)) f(x)m^*(dx)=\frac{1}{s(x_0)-s(l)}\int_l^{I(c)} (s(x)-s(l))m(dx)<\infty, \quad l<c<r,$$
for the finiteness of both $A^*_{\zeta^*}$ and $\zeta$. This is in agreement with the aforementioned corollary.
\end{remark}

\section{Applications}
\subsection{Inversions of Brownian motions killed upon exiting intervals.} Assume that $X$ is a Brownian motion killed upon exiting the interval $E=(l,r)$.  Let $X_0=x_0\in E$.  If $E$ is bounded then we obtain the inversions
\begin{eqnarray}
 I(x)=
 \left\{
\begin{array}{lll} \nonumber

  x_0^2/x, & \mbox{if}\ x_0^2= lr;& \\
  2x_0-x, & \mbox{if}\ 2x_0=l+r;\\
  (x+a)/(b x-1), &\hbox{otherwise},  & \\
\end{array}
\right.
\end{eqnarray}
where
$$
a=\frac{2lr-x_0(l+r)}{x_0^2-lr}x_0 \ \ \hbox{and}\ \
b=\frac{2x_0-(l+r)}{x_0^2-lr}.
$$
 We distinguish three cases appearing in the form of the inversion $I$. In the first case, equation $x_0^2= lr$ implies that $l$ and $r$ are of the same sign. If $l>0$ then
 $X^*$ is the  $3$-dimensional Bessel process killed upon exiting $E$. If $l<0$ then
 $-X^*$ is the $3$-dimensional Bessel process.

 In the second case, $X^*$ is a Brownian motion killed when it exits $E$. In the third case, $X^*$ satisfies the s.d.e. $X_t^*=B_t+x_0+\int_0^t (X^*_s-1/b)^{-1} ds$ for $t<\zeta^*$.
 If $x_0$ is below the arithmetic mean i.e. $x_0<(l+r)/2$ then
 $x_0-1/b=(x_0-r)(x_0-l)/(2x_0-(l+r))$. By uniqueness of the solution to the s.d.e. $R_t=R_0+B_t+\int_0^t (1/R_s) ds$ driving the three-dimensional Bessel process $R$, we get that
 $X^*_t=1/b+R_t$ with $R_0=x_0-1/b$ killed when $R$ exits the interval $(l-1/b, r-1/b)$. If $x_0$ is above the arithmetic mean then we find $X_t^*=1/b-R_{t}$, where $R$ is a three-dimensional Bessel process started at $1/b-x_0$, where $R$ is killed as soon as it exits the interval $( 1/b-r, 1/b-l)$.

If $l$ is finite and $r=\infty$ or $l=-\infty$ and $r$ is finite then by Proposition \ref{kinds-involutions} we respectively obtain
$$I(x)=l+\frac{(x_0-l)^2}{x-l}\ \hbox{and}\ I(x)=r-\frac{(r-x_0)^2}{r-x}.$$
 If $r=+\infty$ then a similar reasoning as above gives that $X^*_t=l+R_t$ where $R$ is a three-dimensional Bessel process started at $x_0-l$.  If $l=-\infty$ we obtain $X^*_t=r-R_t$ where $R$ is a three-dimensional Bessel process started at $r-x_0$.
If $2x_0=l+r$ or $l=-\infty$ and $r=+\infty$ then we obtain the Euclidian reflection in $x_0$, i.e. $x\rightarrow 2x_0-x$, and $X^*$ is a Brownian motion killed when it exits $E$. Note that for $E=(0,\infty)$ the conclusion from our Theorem
\ref{equality-in-distribution} is found in Lemma 3.12 on  p. 257 of
\cite{ry99}. That is $(X_t, t\leq H_0)$ is distributed as
$(1/X^*_{\tau_t^{*}},\, t\leq A_{\infty}^{*})$, where $\tau_t^{*}$ is the
inverse of $A_t^{*}=\int_0^t (X_{s}^{*})^{-4}ds$. In this case $X^*$ is the $3$-dimensional Bessel process; see our last example given in Subsection 5.3 for Bessel processes of other dimensions.  If $E=\mathbb{R}$ then $I$ is the Euclidian reflection in $x_0$. Finally, observe that the set of inversions of $E$ we obtain for the Brownian motion killed when it exits $E$ is precisely
$\mathcal{MI}(E):=\{ I\in \mathcal{MI}\,;\, I(E)=E \}$.

 \subsection{Inversions of drifted Brownian motion and hyperbolic Bessel process of dimension 3.}\label{HypBes}
  Set  $B_t^{(\mu)}=B_t+\mu t$, $t\geq 0$, where $B$ is a
 standard Brownian motion and $\mu \in \mathbb{R}$, $\mu\neq 0$. Thus, $B^{(\mu)}$ is a transient diffusion which drifts to $+\infty$ (resp. to $-\infty$) if $\mu>0$ (resp. $\mu<0$).
 Let us take the reference scale function $s(x)=-e^{-2\mu x}/(2\mu)$. Observe that $s$ is increasing for all $\mu\neq 0$. Moreover
 $\lim_{x\to \infty} s(x) = 0$ if $\mu>0 $ and $\lim_{x\to -\infty} s(x) = 0$ if $\mu<0 $. We take $X$ to be $B^{(\mu)}$ killed when it exits $(l,r)\subseteq \mathbb{R}$. Let us fix $x_0\in E$.

If we take $E=\mathbb{R}$ then by Proposition \ref{kinds-involutions}, even though $X$ is of type 2 if $\mu<0$ and of type 3 if $\mu>0$, the inversion of $E$ in the direction of $s$ is the Euclidian reflection in $x_0$. $X^*$ is the Brownian motion with drift $\mu^*=-\mu$ in this case. If  $s(l)$ and $s(r)$ are finite then using Proposition \ref{Prop-Involutions-0} we obtain appropriate, but in most cases complicated, formulas for $I$.
For Brownian motion with drift the case of the half-line is the most interesting. Take for instance $E=(0,\infty)$ and process $X_t=x_0+B_t+\mu t$ starting from some point $x_0>0$ and killed at zero. We consider two cases:
if $\mu<0$ then $s(0)=-1/(2\mu) > 0$, $s(\infty)=\infty$ and $X$ is of type 2;
if $\mu>0$ then $s(0)=-1/(2\mu) < 0$, $s(\infty)=0$ and $X$ is of type 1.

First, let $\mu<0$. If $X$ is of type 2 then, by Proposition \ref{kinds-involutions}, we have only one possible inversion: $I(x)=s^{-1}\left(s(l)+(s(x_0)-s(l))^2/(s(x)-s(l))\right)$,
which gives
$$
I(x)=
(2|\mu|)^{-1}\ln\left((e^{-2\mu x}-1+(1-e^{-2\mu x_0})^2)/(e^{-2\mu x}-1)\right).
$$
If we choose $x_0=(2|\mu|)^{-1} \ln (1+\sqrt{2})$ then the above formula simplifies to
$$
I(x)=(2|\mu|)^{-1}\ln\left((e^{-2\mu x}+1)/(e^{-2\mu x}-1)\right)=(2|\mu|)^{-1}\ln \coth (|\mu| x).
$$
Now, if $\mu>0$ then $X$ is of type 1. Because $e^{-2\mu x}\neq 0$ implies $s^2(x_0)\neq s(0)s(\infty)=0$ only two cases are possible. Either $2s(x_0)=s(0)+s(\infty)= s(0)$, which gives $x_0=\frac1{2\mu}\ln 2$, and then we have $s$-reflection $I(x)=-\frac{1}{2\mu}\ln(1-e^{-2\mu x})$
or $2s(x_0)\neq s(0)+s(\infty)$ and then the formula from Proposition 3 gives the inversion
$$
I(x)=\frac{1}{2\mu}\ln\left(\frac{1+e^{-2\mu x}(e^{4\mu x_0}-2e^{2\mu x_0})}{1-e^{-2\mu x}}\right).
$$
This in turn simplifies if we choose $e^{4\mu x_0}-2e^{2\mu x_0}=1$, that is, if $x_0=\frac{1}{2\mu}\ln (1+\sqrt{2}).$
Then
$$
I(x)=\frac{1}{2\mu}\ln \left(\frac{1+e^{-2\mu x}}{1-e^{-2\mu x}}\right)=\frac{1}{2\mu} \ln \coth (\mu x).
$$

It is easy to check that if $\mu<0$, then $h(x)=\frac{e^{-2\mu x}-1}{\sqrt{2}}$ and then
$X^*$, being an $h$-process, has generator $L^*f(x)=\frac12 f'' (x)+|\mu| \coth(|\mu| x) f'(x)$.
In particular, if $\mu=-1$ then $I(x)= \frac12 \ln\coth x$ and $X^*$ has generator
$L^*f(x) =\frac12 f ''(x)  +\coth( x) f ' (x)$.
This recovers the well-known fact that $B_t-t$ conditioned to avoid zero
is a three-dimensional hyperbolic Bessel process. The novelty here is that we get $X^*$ as a time changed inversion of $B_t-t$.
If $\mu>0$, then $h(x)=\frac{e^{-2\mu x}+1}{\sqrt{2}}$ and $X^*$ has the generator
$L^*f(x)=\frac12 f'' (x)+\mu \tanh(\mu x) f'(x)$.
 Note that we have shown the following particular result.
\begin{corollary}\label{hypBessel}
 Let  $X$ be a three-dimensional hyperbolic Bessel process on $(0,\infty)$ and  $I(x)=\frac12 \ln\coth x$. Then, $I(X)$ is a time-changed
 drifted Brownian motion $B_t-t$ conditioned to avoid 0. In particular, the functional
 $$A_{\infty}=\frac14\int_0^{\infty} \frac{ds}{ \left(\cosh (X_s) \sinh(X_s) \right)^2}  $$ has the same distribution as the first hitting time of $0$ by the Brownian motion with minus unity drift.  In other words, we have
  $$\mathbb{P}\left(A_{\infty}\in dt\right)= e^{x_0-\frac{t}{2}}\frac{x_0}{\sqrt{2\pi t^3}}e^{-\frac{x_0^2}{2t}}dt, \quad t>0,$$
  where $x_0=\frac{1}{2}\ln (1+\sqrt{2})$.

\end{corollary}

\subsection{Inversions of Bessel processes.}

It is known that two Bessel processes of dimensions $\delta$ and $4-\delta$, respectively, are dual one to another, see e.g. (\cite{Salm-Boro}, \cite{GJ-Y}, \cite{ry99}).
Our focus here is on the construction of the latter dual. To say more, let $X$ be a Bessel process of dimension $\delta\geq 2$. Thus, $E=(0,\,\infty)$, $0$ is polar and  $X$ has the infinitesimal generator $Lf(x)=\frac{1}{2} f''(x) +(\delta-1)/(2x)f'(x)$ for $x\in E$, see (\cite{ry99}, \cite{GJ-Y}). Let  $\nu=(\delta/2)-1$ be the index of $X$. The scale function of $X$ may then be chosen to be
\begin{eqnarray} \nonumber
 s(x)=
 \left\{
\begin{array}{lll}
    -x^{-2\nu}  &\hbox{if}\ \nu>0; \\
  2\log x  &\hbox{if}\ \nu =0;\\
  x^{-2\nu}  &\hbox{otherwise}.
\end{array}
\right.
\end{eqnarray}
$X$ is recurrent if and only if $\delta=2$, see for example Proposition 5.22 on p. 355 of \cite{Karatzas-Shreve}. First, when $\delta=2$ the inversion of $E$ in the direction of $s$ is the s-reflection $x\rightarrow x_0^2/x$. For the other cases, the inversion of $E$ in the direction of $s$ with fixed point 1 is found to be $x\rightarrow 1/x$. Furthermore, $X^*$ is a Bessel process of dimension $4-\delta<2$. Observe that in the two considered cases the involved clock is $A_t=\int_0^t (x_0/X_s)^4\ ds$. Hence, we have shown the following result which is a particular case of Proposition 1.11. on p. 447 of \cite{ry99}.
\begin{corollary}\label{Bessel}
 Let  $X$ be a Bessel process of dimension $\delta\in\mathbb{R}$, killed at $0$ if $\delta<2$, starting from $X_0>0$  and $I(x)=1/x$. Then $I(X)$ is a time-changed
  Bessel process of dimension $4-\delta$. In particular, the functional $\int_0^{\infty}ds/X_s^4$, when $X$ is a Bessel process of dimension $\delta>2$, has the same distribution as the first hitting time of zero by the Bessel process of dimension $4-\delta$.
\end{corollary}

\noindent{\bf Acknowledgment:}   We thank the referee for numerous comments that helped to improve the paper.   We would like to thank Julien Berestycki who asked the first author a  question which led to Corollary \ref{time-reversal}.
  We are greatly indebted to l'Agence Nationale de la Recherche  for  the research grant ANR-09-Blan-0084-01.

\end{document}